\documentclass[12pt]{amsart}

\usepackage{amsfonts}
\usepackage{amsmath}

\usepackage{amssymb,amsmath,epsfig,graphics}

\newcommand{\beq}{\begin{equation}}
\newcommand{\eeq}{\end{equation}}
\newcommand{\bea}{\begin{eqnarray}}
\newcommand{\eea}{\end{eqnarray}}
\newcommand{\beas}{\begin{eqnarray*}}
\newcommand{\eeas}{\end{eqnarray*}}

\def\me{{\mathbb  E}}

\def\mr{{\mathbb  R}}

\def\mp{{\mathbb  P}}

\def\mbj{\mathbf J}

\newtheorem{theorem}{Theorem}[section]
\newtheorem{hypothesis}[theorem]{Hypothesis}
\newtheorem{definition}[theorem]{Definition}
\newtheorem{proposition}[theorem]{Proposition}
\newtheorem{corollary}[theorem]{Corollary}
\newtheorem{lemma}[theorem]{Lemma}
\newtheorem{remark}[theorem]{Remark}
\newtheorem{example}[theorem]{Example}
\newtheorem{examples}[theorem]{Examples}
\newtheorem{foo}[theorem]{Remarks}

\newcommand{\ga}{\gamma}

\title[Gradient bounds for solutions of SDEs driven by fBms]{Gradient bounds for solutions of stochastic differential equations  driven by fractional Brownian motions}

\author{Fabrice Baudoin, Cheng Ouyang}

\address{Department of Mathematics\\Purdue University \\
West Lafayette, IN 47907} \email[Fabrice Baudoin]{fbaudoin@math.purdue.edu}\email[Cheng Ouyang]{couyang@math.purdue.edu}
\thanks{First author supported in part by
NSF Grant DMS 0907326}

\begin{document}

\maketitle

\begin{center}
\textit{To Pr. David Nualart 60th's birthday}
\end{center}
\

\begin{abstract}
We study some functional inequalities satisfied by the distribution of the solution of a stochastic differential equation driven by fractional Brownian motions. Such functional inequalities are obtained through new integration by parts formulas on the path space of a fractional Brownian motion.
\end{abstract}

\tableofcontents

\section{Introduction}

Let $(X^x_t)_{t \ge 0}$ be the solution of a stochastic differential equation
\[ 
X_t=x+\sum_{i=1}^n \int_0^t V_i (X_s^x)  dB^i_s,
\]
where $(B_t)_{t \ge 0}$ is a $n$-dimensional fractional Brownian motion with Hurst parameter $H>\frac{1}{2}$. Under ellipticity assumptions and classical boundedness conditions (see \cite{ba-ha} and \cite{nu-sa}), the random variable $X^x_t$, $t>0$, admits a smooth density with respect to the Lebesgue measure of $\mathbb{R}^n$ and  the functional operator
\[
P_t f(x)=\mathbb{E}\left(f(X_t^x)\right )
\]
is regularizing in the sense that it transforms a bounded Borel function $f$ into a smooth function $P_t f$ for $t>0$. In this note we aim to quantify precisely this regularization property and prove that, under the above assumptions, bounds of the type:

\[
\left| V_{i_1} \cdots V_{i_k} P_t f(x) \right| \le C_{i_1 \cdots i_k}(t,x) \| f \|_\infty, \quad t>0, x \in \mathbb{R}^n,
\]
are  satisfied.  We are moreover able to get an explicit  blow up rate when $t \to 0$: For a fixed $x \in \mathbb{R}^n$, when $t \to 0$,
\[
C_{i_1 \cdots i_k}(t,x) =O\left( \frac{1}{t^{kH}} \right).
\]
Our strategy to prove such bounds is the following. If $f$ is a $C^\infty$ bounded function on $\mathbb{R}^n$, we first prove (see Lemma \ref{th: derivative-semigroup}) that the following commutation holds
\[
V_i P_t f (x)=
\mathbb{E} \left( \sum_{k=1}^n\alpha^i_k(t,x) V_k f (X_t^x) \right),
\]
where the $\alpha (t,x)$'s solve an explicit system of stochastic differential equations. Then, using an integration by parts formula in the path space of the underlying fractional Brownian motion (see Theorem Theorem \ref{th: IBP}) we may rewrite the expectation of the right hand side of the above inequality as $\mathbb{E} \left(\Phi_i (t,x) f (X_t^x) \right)$ where $\Phi_i (t,x)$ is shown to be bounded in $L^p$, $1 \le p <+\infty$ with a blow up rate that may be controlled when $t \to 0$. It yields a bounds on $|V_i P_t f (x)|$. Bounds on higher order derivatives are obtained in a similar way, by iterating the procedure just described. Let us mention here that the bounds we obtain depends on $L^p$ bounds for the inverse of the Malliavin matrix of $X_t^x$. As of today, to the knowledge of the authors such bounds have not yet been obtained in the rough case $H<\frac{1}{2}$. The extension of our results to the case $H<\frac{1}{2}$ is thus not straightforward.

\

We close the paper by an interesting geometric situation where we may prove an optimal and global gradient bound with a constant that is independent from the starting point $x$. In the situation where the equation is driven by a Brownian motion such global gradient bound is usually related to lower bounds on the Ricci curvature of the Riemannian geometry given by the vector fields $V_i$'s, which makes  interesting the fact that the bound also holds with fractional Brownian motions.

\section{ Stochastic differential equations driven by fractional Brownian motions}

We consider the Wiener space of continuous paths:
\[
\mathbb{W}^{n}=\left( \mathcal{C} ( [0,1] , \mathbb{R}^n
), (\mathcal{B}_t)_{0 \leq t \leq 1}, \mathbb{P} \right)
\]
where:
\begin{enumerate}
\item  $\mathcal{C} ( [0,1] , \mathbb{R}^n
)$ is the space of continuous functions $ [0,1] \rightarrow
\mathbb{R}^n$;
\item  $\left( \beta_{t}\right) _{t\geq 0}$ is the coordinate process defined by
$\beta_{t}(f)=f\left( t\right) $, $f \in \mathcal{C} ( [0,1] ,
\mathbb{R}^n )$;
\item  $\mathbb{P}$ is the Wiener measure;
\item  $(\mathcal{B}_t)_{0 \leq t \leq 1}$ is the ($\mathbb{P}$-completed) natural filtration
of $\left( \beta_{t}\right)_{0 \leq t \leq 1}$.
\end{enumerate}
A $n$-dimensional fractional Brownian motion with Hurst parameter
$H \in (0,1)$ is a Gaussian process
\[
B_t = (B_t^1,\ldots,B_t^n), \text{ } t \geq 0,
\]
where $B^1,\ldots,B^n$ are $n$ independent centered Gaussian
processes with covariance function
\[
R\left( t,s\right) =\frac{1}{2}\left(
s^{2H}+t^{2H}-|t-s|^{2H}\right).
\]
It can be shown that such a process admits a continuous version
whose paths are H\"older $\gamma$ continuous, $\gamma<H$.
Throughout this paper, we will always consider the `regular' case,  $H > 1/2$.
In this case the fractional Brownian motion can be constructed on the Wiener space by a
Volterra type representation (see \cite{du}). Namely, under the Wiener measure, the process
\begin{equation}\label{Volterra:representation}
B_t = \int_0^t {K_H}(t,s) d\beta_s, t \geq 0
\end{equation}
is a fractional Brownian motion with Hurst parameter $H$, where
\begin{equation*}
{K_H} (t, s)=c_H s^{\frac{1}{2} - H} \int_s^t
(u-s)^{H-\frac{3}{2}} u^{H-\frac{1}{2}} du \;,\qquad t>s.
\end{equation*}
and $c_H$ is a suitable constant.

Denote by $\mathcal{E}$ the set of step functions on $[0,1]$. Let $\mathcal{H}$ be the Hilbert space defined as the closure of $\mathcal{E}$ with respect to the scalar product
$$\langle\mathbf{1}_{[0,t]},\mathbf{1}_{[0,s]}\rangle_{\mathcal{H}}=R_H(t,s).$$
The isometry $K_H^*$ from $\mathcal{H}$ to $L^2([0,1])$ is given by
$$(K_H^*\varphi)(s)=\int_s^1\varphi(t)\frac{\partial K_H}{\partial t}(t,s)dt.$$
Moreover, for any $\varphi\in L^2([0,1])$ we have
$$\int_0^1\varphi(s)dB_s=\int_0^1(K_H^*\varphi)(s)d\beta_s.$$  


Let us consider for $x \in \mathbb{R}^n$ the solution $(X_t^x)_{t \ge 0}$ of the stochastic differential equation:
\begin{align}\label{equ}
X_t^x=x+\sum_{i=1}^n \int_0^t V_i (X_s^x)  dB^i_s,
\end{align}
where the $V_i$'s are $C^\infty$ bounded vector fields in $\mathbb{R}^n$.
 Existence and uniqueness of solutions for such equations have widely been  studied and are known to hold in this framework (see for instance \cite{nu-ra}). Moreover, the following bounds were proved by Hu and Nualart as an application of fractional calculus methods.
 \begin{lemma}(Hu-Nualart, \cite{HN})\label{HN}
Consider the stochastic differential equation (\ref{equ}). If the derivatives of $V_i$'s are bounded and H\"{o}lder continuous of order $\lambda> 1/H-1$, then
$$\me\left(\sup_{0\leq t\leq T}|X_t|^p\right)<\infty$$
for all $p\geq 2.$ If furthermore $V_i$'s are bounded and $\me(\exp(\lambda|X_0|^q))<\infty$ for any $\lambda>0$ and $q<2H$, then
$$\me\left(\exp \lambda\left(\sup_{0\leq t\leq T}|X_t|^q\right)\right)<\infty$$
for any $\lambda >0$ and $q<2H$.
\end{lemma}

\

Throughout our discussion, we assume that the following assumption is in force:
\begin{hypothesis}\label{H1}
\ \\
(1) $V_i(x)$'s are bounded smooth vector fields on $\mathbb{R}^n$ with bounded derivatives at any order. 

(2) For every $x \in \mathbb{R}^n$, $(V_1(x),\cdots,V_n(x))$ is a basis of $\mathbb{R}^n$. 

\end{hypothesis}

Therefore, in this framework, we can find functions
$\omega^k_{ij}$ such that
\begin{align}\label{omega}
[V_i,V_j]=\sum_{k=1}^n \omega_{ij}^k V_k,
\end{align}
where the $\omega_{ij}^k$'s are bounded smooth functions on $\mr^n$ with bounded derivatives at any order.

\section{Integration by parts formulas}

We first introduce notations and basic relations for the purpose of our discussion.
Consider the diffeomorphism $\Phi(t,x)=X^x_t:\mr^n\rightarrow\mr^n$. Denote by 
$\mathbf{J_t}=\frac{\partial X^x_t}{\partial x}$
the Jacobian of $\Phi(t,\cdot)$. It is standard (see \cite{nu-sa} for details) that
\begin{align}\label{equ-J}
d\mbj_t=\sum_{i=1}^n\partial V_i(X^x_t)\mbj_t dB^i_t,\quad\quad\quad\mathrm{with}\ \ \mbj_0=\mathbf{I},
\end{align}
and
\begin{align}\label{equ-J-1}
d\mbj^{-1}_t=-\sum_{i=1}^n\mbj^{-1}_t\partial V_i(X^x_t)dB^i_t,\quad\quad\quad\mathrm{with}\ \ \mbj^{-1}_0=\mathbf{I}.
\end{align}
For any $C^\infty_b$ vector field $W$ on $\mr^n$, we have that
$$({\Phi_t}_*W)(X_t^x)=\mbj_tW(x),\quad\quad\mathrm{and}\quad({\Phi_t}^{-1}_*W)(x)=\mbj^{-1}_tW(X_t^x).$$
Here ${\Phi_t}_*$ is the push-forward operator with respect to the diffeomorphism $\Phi(t,x):\mr^n\rightarrow\mr^n.$
Introduce the non-degenerate $n\times n$ matrix value process \begin{align}\label{alpha}\alpha(t,x)=(\alpha^i_j(t,x))_{i,j=1}^n\end{align} by
\begin{align*}
({\Phi_t}_*V_i)(X^x_t)=\mbj_t(V_i(x))=\sum_{k=1}^n \alpha^i_k(t,x)V_k(X_t^x)\quad\quad\quad i=1,2,...,n.
\end{align*}
Note that $\alpha(t,x)$ is non-degenerate since we assume $V_i$'s form a basis at each point $x\in\mr^n$. Denote by \begin{align}\label{beta}\beta(t,x)=\alpha^{-1}(t,x).\end{align}
Clearly we have
\begin{align}\label{rel-beta}({\Phi_t}_*^{-1}V_i(X_t^x))(x)=\mbj_t^{-1}V_i(X_t^x)(x)=\sum_{k=1}^n \beta^i_k(t,x)V_k(x)\quad\quad\quad i=1,2,...,n.\end{align}

\begin{lemma}\label{th: equ-alpha-beta}
Let $\alpha(t,x)$ and $\beta(t,x)$ be as above, we have
\begin{align}\label{equ-alpha}
d\alpha^i_j(t,x)=-\sum_{k,l=1}^n\alpha^i_k(t,x)\omega^j_{lk}(X_t^x)dB^{l}_t,\quad\quad\mathrm{with}\quad \alpha^i_j(0,x)=\delta^i_j;
\end{align}
and
\begin{align}\label{equ-beta}
d\beta^i_j(t,x)=\sum_{k,l=1}^n\omega^k_{li}(X_t^x)\beta^k_j(t,x)dB^l_t,\quad\quad\mathrm{with}\quad\quad \beta^i_j(0,x)=\delta^i_j.
\end{align}
\end{lemma}

\begin{proof}
The initial values are apparent by the definition of $\alpha$ and $\beta$. We show how to derive equation (\ref{equ-beta}). Once the equation for $\beta(t,x)$ is obtained, it is standard to obtain that  $\alpha(t,x)=\beta^{-1}(t,x)$. 

Consider the $n\times n$ matrix $V=(V_1,V_2,...,V_n)=(V^i_j)$ obtained from the vector fields $V$.
Let $W$ be the inverse matrix of $V$. 
It is not hard to see we have
$$\beta^i_j(t,x)=\sum_{k=1}^nW^j_k(x)(\mbj_t^{-1}V_i(X_t^x))^k(x).$$
By the equation for $X_t^x$,  relation (\ref{omega}), equation (\ref{equ-J-1}),  and It\^{o}'s formula, we obtain
\begin{align*}
d(\mbj^{-1}_tV_i(X_t^x))(x)&=\sum_{k=1}^n(\mbj^{-1}_t[V_k,V_i](X_t^x))(x)dB_t^k\\
&=\sum_{k,l=1}^n\omega^l_{ki}(X_t^x)(\mbj^{-1}_tV_l(X_t^x))(x)dB_t^k.
\end{align*}
Hence
\begin{align*}
d\beta^i_j(t,x)=\sum_{k,l=1}^n\omega^l_{ki}(X_t^x)\beta^l_j(t,x)dB^k_t.
\end{align*}
This completes our proof.
\end{proof}

Define now $h_i(t,x): [0,1]\times\mr^n\rightarrow \mathcal{H}$ by
\begin{align}\label{def-h}
h_i(t,x)=(\beta^k_i(s,x)\mathbb{I}_{[0,t]}(s))_{k=1,...,n},\quad\quad i=1,...,n. \end{align}

Introduce $M_{i,j}(t,x)$  given by
\begin{align}\label{def-M}
M_{i,j}(t,x)=\frac{1}{t^{2H}}\langle h_i(t,x), h_j(t,x)\rangle_{\mathcal {H}}.
\end{align}

For each $t\in [0,1]$, consider the semi-norms
$$\|f\|_{\gamma,t}:=\sup_{0\leq v<u\leq t}\frac{|f(u)-f(v)|}{(u-v)^\gamma}.$$
The semi-norm $\|f\|_{\gamma, 1}$ will simply be denoted by $\|f\|_\gamma$.

We have the following two important estimates.
\begin{lemma}\label{th: est-ab}
Let $\alpha(t,x)$, $\beta(t,x)$ and $h_i(t,x)$ be as above.  We have:


(1) For any multi-index $\nu$, integers $k, p\geq 1$, there exists a constant $C_{k,p}(x)>0$ depending on $k, p$ and $x$ such that for all $x\in \mr^n$
$$\sup_{0\leq t\leq 1}\left\|\frac{\partial^{|\nu|}}{\partial x^\nu}\alpha(t,x)\right\|_{k,p}<C_{k,p}(x),\quad \sup_{0\leq t\leq 1}\left\|\frac{\partial^{|\nu|}}{\partial x^\nu}\beta(t,x)\right\|_{k,p}<C_{k,p}(x).$$

(2) For all integers $k, p\geq 1$, $\delta h_i(t,x)\in\mathbb{D}^{k,p}$.  Moreover, there exists a constant $C_{k,p}(x)$ depending on $k, p$ and $x$ such that
$$\|\delta h_i(t,x)\|_{k,p}<C_{k,p}(x)t^{H},\quad\quad t\in[0,1].$$
In the above $\delta$ is the adjoint operator of $\mathbf{D}$.
\end{lemma}
\begin{proof}
The result in (1) follows from equation (\ref{equ-alpha}), (\ref{equ-beta}) and Lemma \ref{HN}. In what follows, we show (2). Note that we have (c.f. Nualart\cite{Nu06}) 
\begin{align*}\delta h_i(t,x)&=\int_0^1 h_i(t,x)_udB_u-\alpha_H\int_0^1\int_0^1\mathbf{D}_uh(t,x)_v|u-v|^{2H-2}dudv\\
&=\int_0^t \beta_i(u,x)dB_u-\alpha_H\int_0^t\int_0^t\mathbf{D}_u\beta_i(v,x)|u-v|^{2H-2}dudv.
\end{align*}
Here $\alpha_H=H(2H-1)$. From the above representation of $\delta h_i$ and the result in (1), it follows immediately that $\delta h_i(t,x)\in \mathbb{D}^{k,p}$ for all integers $k,p\geq 1$.  To show
$$\|\delta h_i(t,x)\|_{k,p}<C_{k,p}(x)t^H\quad\quad\mathrm{ for\ all}\ t\in[0,1],$$
it suffices to prove
$$\left|\int_0^t\beta_i(u,x)dB_u\right|\leq C(x)t^H\quad\quad t\in[0,1].$$
Here $C(x)$ is a random variable in $L^p(\mp)$.
Indeed, by standard estimate, we have 
$$\left\|\int_0^\cdot (\beta_i(u,x)-\beta_i(0,x))dB_u\right\|_{\gamma, t}\leq C\|\beta(\cdot,x)\|_{\tau,t}\|B\|_{\gamma,t},\quad\quad t\in[0,1].$$
In the above $ \frac{1}{2}<\tau, \gamma<H$ and $\tau+\gamma>1$, and $C>0$ is a constant only depending on $\gamma$. Therefore
$$\left|\int_0^t \beta_i(u,x)dB_u\right|\leq C\|\beta(\cdot,x)\|_{\tau,t}\|B\|_{\gamma,t}t^\gamma+|\beta(0,x)||B_t|,\quad\quad t\in[0,1].$$
Together with the fact that for any $\tau<H$, there exists a random variable $G_\tau(x)$ in $L^p(\mp)$ for all $p>1$ such that
$$|\beta(t,x)-\beta(s,x)|< G_\tau(x)|t-s|^\tau,$$


the proof is now completed. 
\end{proof}

\begin{lemma}\label{th: est-M}
Let $M(t,x)=(M_{i,j}(t,x))$ be given in (\ref{def-M}). We have for all $p\geq 1$,
$$\sup_{t\in[0,1]}\me \left[ \det (M(t,x))^{-p} \right] <\infty.$$
\end{lemma}
\begin{proof}
Denote the Malliavin matrix of $X_t^x$ by $\Gamma(t,x)$. By definition
\begin{align*}
\Gamma_{i,j}(t,x)=\langle \mathbf{D}_sX_t^i, \mathbf{D}_sX_t^j\rangle_{\mathcal{H}}=\alpha_H\int_0^t\int_0^t \mathbf{D}_uX_t^i \mathbf{D}_vX_t^j|u-v|^{2H-2}dudv.
\end{align*}
It can be shown that for all $p>1$ (cf. Baudoin-Hairer \cite{ba-ha}, Hu-Nualart \cite{HN} and Nualart-Saussereau \cite{nu-sa}),
\begin{align}\label{est-Gamma}
\sup_{t\in[0,1]}\me\left(\det \frac{\Gamma(t,x)}{t^{2H}}\right)^{-p}<\infty.
\end{align}

Introduce $\gamma$ by
$$\gamma_{i,j}(t,x)=\alpha_H\int_0^t\int_0^t\sum_{k=1}^n\left(\mathbf{J}_u^{-1}V_k(X_u)\right)^i\left(\mathbf{J}_v^{-1}V_k(X_v)\right)^j|u-v|^{2H-2}dudv.$$
Since $\mathbf{D}_s^kX_t=\mathbf{J}_t \mathbf{J}_s^{-1}V_k(X_s)$, we obtain
\begin{align}\label{rep-Gamma}
\Gamma(t,x)=\mathbf{J}_t\gamma(t,x) \mathbf{J}_t^T.
\end{align}

Recall
\begin{align*}
M_{i,j}(t,x)=\frac{1}{t^{2H}}\langle h_i(t,x), h_j(t,x)\rangle_{\mathcal {H}},
\end{align*}
where
\begin{align*}
h_i(t,x)=(\beta^k_i(s,x)\mathbb{I}_{[0,t]}(s))_{k=1,...,n},\quad\quad i=1,...,n. 
\end{align*}
By (\ref{rel-beta}) and (\ref{rep-Gamma}), we have
\begin{align}\label{rep-M}
V(x)M(t,x)V(x)^T=\frac{1}{t^{2H}} \gamma(t,x)=\mathbf{J}^{-1}\frac{\Gamma(t,x)}{t^{2H}}(\mathbf{J}^{-1})^T.
\end{align}
Finally, by equation (\ref{equ-J}), Lemma \ref{HN} and estimate (\ref{est-Gamma}) we have for all $p\geq 1$
$$\sup_{t\in[0,1]}\me \left[ \det (M_{i,j})^{-p} \right]<\infty,$$
which is the desired result.
\end{proof}

The following definition is inspired by Kusuoka \cite{Kusuoka}.
\begin{definition}
Let $H$ be a separable real Hilbert space and $r\in \mr$ be any real number. Introduce $\mathcal{K}_r(H)$ the set of mappings $\Phi(t,x): (0,1]\times\mr^n\rightarrow \mathbb{D}^\infty(H)$ satisfying:

(1) $\Phi(t,x)$ is smooth in $x$ and $\frac{\partial^\nu\Phi}{\partial x^\nu}(t,x)$ is continues in $(t,x)\in(0,1]\times\mr^n$ with probability one for any multi-index $\nu$;

(2) For any $n, p>1$ we have
$$\sup_{0< t\leq 1}t^{-rH}\left\|\frac{\partial^\nu \Phi}{\partial^\nu x}(t,x)\right\|_{\mathbb{D}^{k,p}(H)}<\infty.$$
\end{definition}
We denote $\mathcal{K}_r(\mr)$ by $\mathcal{K}_r$.

\begin{lemma}\label{th: K}
With probability one, we have

(1) $\alpha(t,x), \beta(t,x)\in \mathcal{K}_0$;

(2) $\delta h_i(t,x)\in \mathcal{K}_1$;

(3) Let $(M^{-1}_{i,j})$ be the inverse matrix of $(M_{i,j})$. Then $M^{-1}_{i,j}\in \mathcal{K}_0$ for all $i,j=1,...,n$.
\end{lemma}
\begin{proof}
The first two statements are immediate consequences of Lemma \ref{th: est-ab}.  The third statement  follows by writing $M^{-1}=\frac{\mathrm{adj}M}{\det M}$,  estimates in Lemma \ref{th: est-ab} (1) and Lemma \ref{th: est-M}. \end{proof}

Now we can state one of our main results in this note.

\begin{theorem}\label{th: IBP}
Let $f$ be any $C^\infty$ bounded  function and $\Phi(t,x): \Omega\rightarrow\mathcal{K}_r$ we have
$$\me\left(\Phi(t,x)V_if(X_t^x)\right)=\me\left((T^*_{V_i}\Phi(t,x)) f(X_t^x)\right),$$ 
where $T^*_{V_i}\Phi(t,x)$ is an element in $\mathcal{K}_{r-1}$ with probability one.
\end{theorem}
\begin{proof}
This is primarily integration by parts together with the estimates obtained before. First note
\begin{align*}
\mathbf{D}^j_s f(X_t)&=\langle \nabla f(X_t), \mathbf{D}_s^j X_t\rangle\\
&=\langle\nabla f(X_t), \mathbf{J}_t\mathbf{J}_s^{-1} V_j(X_s)\rangle\\
&=\sum_{k,l=1}^n h^j_k(t)\alpha_l^k(t)(V_l f)(X_t).
\end{align*}
Hence
\begin{align}\label{rep-Vf}
V_if(X_t)=\frac{1}{t^{2H}} \sum_{j,l=1}^n \beta^i_j(t)M_{jl}^{-1}\langle \mathbf{D} f(X_t), h_l(t)\rangle_{\mathcal{H}}.
\end{align}
Therefore, we have
\begin{align*}
&\me \left(\Phi(t,x) V_if(X_t)\right)\\=&\frac{1}{t^{2H}} \sum_{k,l=1}^n \me \left(\langle \mathbf{D} f(X_t), \Phi(t,x)\beta^i_k(t)M_{kl}^{-1}(t)h_l(t)\rangle_{\mathcal{H}} \right)\\
=&\frac{1}{t^{2H}}\sum_{k,l=1}^n\me \left( \big[\delta\left(\Phi(t,x)\beta^i_k(t)M_{kl}^{-1}(t)h_l(t)\right)\big] f(X_t)\right)\\
=&\sum_{k,l=1}^n \me \left( \bigg[\frac{1}{t^{2H}}\Phi(t,x)\beta_k^i(t)M_{kl}^{-1}(t)\delta h_l(t)-\frac{1}{t^{2H}}\langle \mathbf{D}(\Phi(t,x)\beta_k^i(t)M_{kl}^{-1}(t)), h_l(t)\rangle_{\mathcal{H}}\bigg]f(X_t)\right).
\end{align*}
By Lemma \ref{th: K}, the first term in the brackets above is in $\mathcal{K}_{r-1}$ and the second term is in $\mathcal{K}_r$. Finally, denote 
$$T^*_{V_i}\Phi(t,x)=\sum_{k,l=1}^n  \bigg[\frac{1}{t^{2H}}\Phi(t,x)\beta_k^i(t)M_{kl}^{-1}(t)\delta h_l(t)-\frac{1}{t^{2H}}\langle \mathbf{D}(\Phi(t,x)\beta_k^i(t)M_{kl}^{-1}(t)), h_l(t)\rangle_{\mathcal{H}}\bigg].$$
It is clear that $T^*_{V_i}\Phi(t,x)\in \mathcal{K}_{r-1}$. The proof is completed.
\end{proof}

\section{Gradient Bounds}

With the  integration by parts formula of Theorem \ref{th: IBP} in hands we can now prove our gradient bounds. We start with the following basic commutation formula:

\begin{lemma}\label{th: derivative-semigroup}
For $i=1,2,...,n$, we have the commmutation

\[
V_i P_t f (x)=\mathbb{E} \left( \left( (\mbj_tV_i) f \right)(X_t^x)\right)=
\mathbb{E} \left( \sum_{k=1}^n\alpha^i_k(t,x) V_k f (X_t^x) \right),
\]
where the $\alpha (t,x)$ solve the system of stochastic differential equations (\ref{equ-alpha}).
\end{lemma}
\begin{proof}
For any $C^\infty_b$-vector field $W$ on $\mr^n$ we have
\begin{align*}
WP_tf(x)=\me  \left( ((\mbj_t W)f)(X_t^x)\right).
\end{align*}
The remainder of the proof is then clear from the computations in the previous section.
\end{proof}

Finally we have the following gradient bounds.

\begin{theorem}\label{th: gen-grad-bound}
Let $p>1$. For $i_1,...,i_k \in \{1,...,n\}$, and $x \in ,\mathbb{R}^n$, we have
\[
\left| V_{i_1}...V_{i_k} P_t f (x) \right| \le C(t,x) ( P_t f^p(x) )^{\frac{1}{p}} \quad t \in [0,1]
\]
with  $C(t,x)=O\left(\frac{1}{t^{Hk}}\right)$ when $t \to 0$.
\end{theorem}
\begin{proof}
By Theorem \ref{th: IBP} and Lemma \ref{th: derivative-semigroup},  for each $k\geq 1$ there exists  a $\Phi^{(-k)}(t,x)\in\mathcal{K}_{-k}$ such that

$$V_{i_1}...V_{i_k}P_t f(x)=\me  \left( \Phi^{(-k)}(t,x)f(X_t) \right).$$
Now an application of H\"{o}lder's inequality gives us the desired result.
\end{proof}

\begin{remark}
Here let us emphasize a simple but important consequence of the above theorem that, suppose $f$ is uniformly bounded, then
\[
\left| V_{i_1}...V_{i_k} P_t f (x) \right| \le C(t,x) \| f \|_\infty \quad t \in [0,1]
\]
where  $C(t,x)=O\left(\frac{1}{t^{Hk}}\right)$ as $t \to 0$.
\end{remark}

Another direct corollary of Theorem \ref{th: gen-grad-bound} is the following inverse Poincar\'{e} inequality. 
\begin{corollary}
For $i_1,...,i_k \in \{1,...,n\}$, and $x \in \mathbb{R}^n$,
\[
\left| V_{i_1}...V_{i_k} P_t f (x) \right|^2 \le C(t,x) ( P_t f^2(x) -(P_t f)^2 (x))  \quad t \in [0,1]
\]
with  $C(t,x)=O\left(\frac{1}{t^{2Hk}}\right)$ when $t \to 0$.
\end{corollary}

\begin{proof}
By Theorem \ref{th: gen-grad-bound}, for any constant $C\in\mr$ we have
\[
\left| V_{i_1}...V_{i_k} P_t f (x) \right|^2=\left| V_{i_1}...V_{i_k} P_t (f-C) (x) \right|^2  \le C(t,x) ( P_t (f-C)^2(x) ) \quad t \in [0,1]
\]
with  $C(t,x)=O\left(\frac{1}{t^{2Hk}}\right)$ when $t \to 0$. Now minimizing $C\in \mr$ gives us the desired result.
\end{proof}

\begin{remark}
For each smooth function $f:\mathbb{R}^n \rightarrow \mathbb{R}$, denote
\[
\Gamma (f)=\sum_{i=1}^n (V_i f)^2.
\]
We also have, for $i_1,...,i_k \in \{1,...,n\}$, and $x \in \mathbb{R}^n$,
\[
\left| V_{i_1}...V_{i_k} P_t f (x) \right|^2 \le C(t,x)  P_t \Gamma (f)(x), \quad t \in [0,1]
\]
with   $C(t,x)=O\left(\frac{1}{t^{2H(k-1)}}\right)$ when $t \to 0$.  Indeed, by Theorem \ref{th: IBP} and Lemma \ref{th: derivative-semigroup}, we know that for each $k\geq 1$, there exists  $\Phi^{(1-k)}_j(t,x)\in \mathcal{K}_{1-k}, j=1,2,...,n$ such that
$$V_{i_1}...V_{i_k} P_tf(x)=\me \left( \Phi^{(1-k)}_{j}(t,x) (V_jf)(X_t)\right).$$
The sequel of the argument is then clear.
\end{remark}

\section{A global gradient bound }

Throughout our discussion in this section, we show that under some additional conditions on the vector fields $V_i,..., V_n$, we are able to obtain
$$ \sqrt{\Gamma(P_tf)}\leq  P_t(\sqrt{\Gamma (f)}),$$
uniformly in $x$, where we denoted as above

\[
\Gamma (f)=\sum_{i=1}^n (V_i f)^2.
\]

  For this purpose, we need the following additional structure equation imposed on vector fields $V_i,..., V_d$.
\begin{hypothesis}\label{H2} In addition to Hypothesis \ref{H1}, we assume
the smooth and bounded functions $\omega_{ij}^k$ satisfy:
\begin{align*}\label{eq:antisym-V}
\omega_{ij}^k =-\omega_{ik}^j,\quad\quad\quad 1\leq i,j,k\leq d.
\end{align*}
\end{hypothesis}

Interestingly, such an assumption already appeared in a previous work of the authors \cite{ba-ou} where they proved an asymptotic expansion of the density of $X_t$ when $t \to 0$.

\begin{remark}
In the case of a stochastic differential equation driven by a Brownian motion, the functional operator $P_t$ is a diffusion semigroup with infinitesimal generator $L=\frac{1}{2}\left( \sum_{i=1}^n V_i^2\right)$. The gradient subcommutation
\[
\sqrt{\Gamma(P_tf)}\leq  P_t(\sqrt{\Gamma f}),
\]
is then known to be equivalent to the fact that the Ricci curvature of the Riemannian geometry given by the vector fields $V_i$'s is non negative (see for instance \cite{bakry-stflour}).
\end{remark}

The following approximation result, which can be found for instance in \cite{FV-bk}, will also be used in the sequel:
\begin{proposition}\label{prop:lin-interpol}
For $m\ge 1$, let
$B^{m}=\{B^m_t;\, t\in [0,1]\}$ be the sequence of linear interpolations of
$B$ along the dyadic subdivision of $[0,1]$ of mesh $m$; that
is if $t_i^m= i 2^{-m} $ for $i=0,..., 2^m;$ then for $t \in
(t_i^m, t_{i+1}^m ],$
\begin{equation*}
B^{m}_t=B_{t_{i^m}} +
\frac{t-t_{i^m}}{t_{i+1}^m - t_i^m} (B_{t_{i+1}^m}
-B_{t_{i}^m} ).
\end{equation*}
Consider $X^{m}$ the solution to equation (\ref{equ}) restricted to $[0,1]$, where $B$ has been replaced by $B^{m}$. 
Then almost surely, for any $\ga<H$ and $t\in[0,1]$ the following holds true:
\begin{equation}\label{eq:limit-interpol}
\lim_{m\rightarrow\infty}  \|X^{x}-X^{m}\|_{\gamma} =0 
\end{equation}
\end{proposition}

\begin{theorem}\label{th: inf-bound-alpha}Recall the definition of $\alpha(t,x)$ in (\ref{alpha}) and 
\begin{align}\label{equ: inf-bound-alpha}
d\alpha^i_j(t,x)=-\sum_{k,l=1}^n \alpha^i_k(t,x)\omega^j_{lk}(X_t^x)dB^{l}_t,\quad\quad\mathrm{with}\quad \alpha^i_j(0,x)=\delta^i_j.
\end{align}
Under Assumption \ref{H2},  uniformly in $t$ and $x$, we have
\begin{equation}\label{eq:bnd-first-deriv-infty}
\left\| \sum_{j=1}^n \alpha^i_j(t,x)^2 \right\|_{L^\infty}    \le 1;\quad\quad  \mathrm{and}\ \ \left\| \sum_{i=1}^n \alpha^i_j(t,x)^2 \right\|_{L^\infty}    \le 1,
\end{equation}
almost surely.
\end{theorem}

\begin{proof}
Let us thus consider  $X^m_t$ and $\alpha^m(t,x)$ the solution of (\ref{equ}) and (\ref{equ: inf-bound-alpha}) where $B$ is
replaced by $B^{m}$, that is
\begin{align*}
&dX_t^m= \sum_{i=1}^n V_i (X_s^m)  dB^{m,i}_s,\\
&d\alpha^m(t,x)=- \sum_{k=1}^n\alpha^m(t,x) \omega_{k} (X_s^m) dB^{m,k}_s,
\end{align*}
with $X_0=x$ and $\alpha(0,x)=I$. Here $\omega_k=(\omega_{kj}^i)$. In order to show that  the process $\alpha(t,x)$ is uniformly bounded, by applying Proposition \ref{prop:lin-interpol} to the couple $(X, \alpha)$,   it is sufficient to prove our uniform bounds on $\alpha^m(t,x)$, uniformly in $m$.
In the sequel set 
\begin{equation*}
\Delta B_{{t_{n-1}^{m}}t_{n}^{m}}^{k,m}:=
\frac{B^{k,m}_{t_{n}^{m}}-B^{k,m}_{t_{n-1}^{m}} }{t_{n}^m - t_{n-1}^m},
\quad\mbox{for}\quad 1\le n\le 2^m \mbox{ and } 1\le k\le n.
\end{equation*}
Then, for 
$t \in [t_{n-1}^m, t_{n}^m )$, we have
\[
d\alpha^m(t,x)=-\alpha^m(t,x) \sum_{k=1}^n \omega_{k} (X_{t}^m) \, \Delta B_{{t_{n-1}^{m}}t_{n}^{m}}^{k,m} dt ,
\]
Therefore, for $t \in [t_{n-1}^m, t_{n}^m )$, we obtain
\[
\alpha^m(t,x)=\left(e^{-\sum_{k=1}^n\Delta B_{{t_{n-1}^{m}}t_{n}^{m}}^{k,m}\int_{t^m_{n-1}}^t \omega_{k} (X_{s}^m) ds }\right)\alpha^m({t_{n-1}^m,x})
\]
Proceeding inductively, we end up with the following identity, valid for $t \in [t_{n-1}^m, t_{n}^m )$ and
$n=0,...,2^m$:
\begin{multline}\label{un}
\alpha^m(t,x) 
=e^{- \sum_{k=1}^n\Delta B_{{t_{n-1}^{m}}t_{n}^{m}}^{k,m} \int_0^t\omega_{k} (X_{s}^m)ds } \times
\cdots \times 
e^{- \sum_{k=1}^n\Delta B_{{t_{0}^{m}}t_{1}^{m}}^{k,m} \int_{t_0^m}^{t_1^m}\omega_{k} (X_{s}^m)ds}.
\end{multline}
By Assumption \ref{H2}, each $\omega_k$ is a skew-symmetric matrix,  expression (\ref{un}) gives us
\begin{align*}
 \left\| \sum_{j=1}^n \alpha^{m,i}_j(t,x)^2 \right\|_{L^\infty}    \le 1;\quad\quad  \mathrm{and}\ \ \left\| \sum_{i=1}^n \alpha^{m,i}_j(t,x)^2 \right\|_{L^\infty}    \le 1.
\end{align*}
 This is our claimed uniform bound on $\alpha^m(t,x)$, from which the end of our proof is easily deduced.

\end{proof}

As a direct consequence of Lemma \ref{th: derivative-semigroup} and Theorem \ref{th: inf-bound-alpha}, we have the main result of this section.

\begin{theorem}
Under Assumption \ref{H2}, we have uniformly in $x$
$$\sqrt{\Gamma(P_tf)}\leq P_t(\sqrt{\Gamma (f)}).$$
\end{theorem}
\begin{proof}
By applying Lemma \ref{th: derivative-semigroup},  Cauchy-Schwarz inequality and then Theorem \ref{th: inf-bound-alpha}, we have for any vector $a=(a_i)\in\mr^n$
\begin{align*}
\sum_{i=1}^na_iV_i P_t f (x)&=
\mathbb{E} \left( \sum_{i,k=1}^na_i\alpha^i_k(t,x) V_k f (X_t^x) \right)\\
&\leq \me \left[ \left(\sum_{k=1}^n\left(\sum_{i=1}^na_i\alpha_k^i(t,x)\right)^2\right)^{\frac{1}{2}}\left(\sum_{k=1}^n (V_kf(X_t^x))^2\right)^{\frac{1}{2}} \right] \\
&\leq \|a\|\me\left(\sum_{k=1}^n V_kf(X_t^x)^2\right)^{\frac{1}{2}}.
\end{align*}
By choosing
$$a_i=\frac{(V_i P_t f) (x) }{\sqrt{  \sum_{i=1}^n  (V_i P_t f)^2(x) }},$$
we obtain
$$\sqrt{\Gamma(P_tf)}=\sqrt{\sum_{i=1}^n (V_iP_tf)^2}\leq   P_t(\sqrt{\Gamma (f)}).$$
The proof is completed.
\end{proof}
\begin{remark}
Since $P_t$ comes from probability measure, we observe from Jensen inequality that
$$\sqrt{\Gamma(P_tf)}\leq  P_t(\sqrt{\Gamma (f)})$$
implies
$${\Gamma(P_tf)}\leq  P_t(\Gamma f).$$
\end{remark}

\end{document}